\documentclass[12pt]{article}

\usepackage{amsmath, amssymb, amsfonts, amsbsy, amscd, latexsym, amsthm}
\date{}

\newtheorem{Theorem}{Theorem}[section]

\newtheorem{Corollary}[Theorem]{Corollary}
\newtheorem{Lemma}[Theorem]{Lemma}

\theoremstyle{definition}
\newtheorem{Definition}[Theorem]{Definition}

\theoremstyle{remark}

\numberwithin{equation}{section}

\title{A relation for a class of Racah polynomials}

\author{Ilia D. Mishev
\footnote{Department of Mathematics, University of Colorado at Boulder,
Campus Box 395, Boulder, CO 80309-0395, U.S.A.
E-mail address: ilia.mishev@colorado.edu}}

\begin{document}

\maketitle

\begin{abstract}

In this paper we derive a relation for a class of Racah polynomials that appear in a conjecture of Kresch and Tamvakis. 
The relation follows from an inversion formula for a transformation of a discrete sequence of complex
numbers $\{ x_n \}_{n=0}^{\infty}$. As a result of our inversion formula, we also obtain other
combinatorial identities.

\end{abstract}

\section{Introduction}

Let $p$ and $q$ be non-negative integers. Let $a_1, a_2, \ldots, a_p, b_1, b_2, \ldots, b_q, z \in \mathbb{C}$.
The hypergeometric series of type ${}_pF_q$ with numerator parameters $a_1, a_2, \ldots, a_p$ and 
denominator parameters $b_1, b_2, \ldots, b_q$ is defined by
\begin{equation}
\label{210}
 {}_pF_q \left[ {\displaystyle a_1,a_2,\ldots,a_p;
\atop \displaystyle b_1,b_2,\ldots,b_q;} z\right] =
\sum_{n=0}^{\infty} \frac{(a_1)_n(a_2)_n \cdots
(a_p)_n}{n!(b_1)_n(b_2)_n \cdots (b_q)_n}z^n,
\end{equation}
where the rising factorial $(a)_n$ is given by
\begin{equation*}
(a)_n=\left\{ \begin{array}{ll}
a(a+1)\cdots(a+n-1), & n>0,\\
1, & n=0.
\end{array} \right.
\end{equation*}

If no numerator parameter is a non-positive integer, we need no denominator parameter to be a non-positive integer.
In this case, the series in (\ref{210}) converges absolutely for all $z$ if $p<q+1$. If $p>q+1$, the series
converges only when $z=0$. In the case $p=q+1$, the series converges absolutely if $|z|<1$ or if
$|z|=1$ and $\textrm{Re}(\sum_{i=1}^qb_i-\sum_{i=1}^pa_i)>0$ 
(see \cite[p.\ 8]{Ba}).

If a numerator parameter is a non-positive integer, then, letting $-n$ be the largest non-positive integer
numerator parameter, only the first $n+1$ terms of the series (\ref{210}) are non-zero and the series
is said to terminate. In this case, we require that no denominator parameter be in the set
$\{-n+1, -n+2, \ldots \}$. We note that (\ref{210}) reduces to a polynomial in $z$ of degree $n$. 

When $z=1$, we say that the series is of unit argument and of type
${}_pF_q(1)$. If $\sum_{i=1}^qb_i-\sum_{i=1}^pa_i=1$, the
series is called Saalsch\"utzian.

We will make use of the Chu-Vandermonde formula 
(see \cite[p.\ 3]{Ba})
for the sum of a terminating ${}_2F_1(1)$ series:
\begin{equation}
\label{220}
{}_2F_1 \left[ {\displaystyle -n,a;
\atop \displaystyle b;} 1\right] 
=\frac{(b-a)_n}{(b)_n}.
\end{equation}

We will also use the binomial coefficient identities
\begin{equation}
\label{225}
\binom{n+1}{k}=\binom{n}{k}+\binom{n}{k-1},
\quad 1 \leq k \leq n,
\end{equation}
\begin{equation}
\label{230}
\binom{n}{k} \binom{k}{m} = \binom{n}{m} \binom{n-m}{k-m},
\quad 0 \leq m \leq k \leq n,
\end{equation}
and
\begin{equation}
\label{233}
\frac{\binom{m+k}{m}}{\binom{n+k}{n}}
=\frac{\binom{n}{m}}{\binom{n+k}{m+k}},
\quad 0 \leq m \leq n, \quad k \geq 0. 
\end{equation}

The Racah polynomials, first given by Wilson \cite{Wil}, are defined by
(see also \cite{KS})
\begin{eqnarray}
\label{240}
&&R_n(\lambda(x);\alpha,\beta,\gamma,\delta) \nonumber \\
&&={}_4F_3 \left[ {\displaystyle -n,n+\alpha+\beta+1,-x,x+\gamma+\delta+1;
\atop \displaystyle \alpha+1,\beta+\delta+1,\gamma+1;} 1\right], \\
&&n=0,1,2, \ldots, N, \nonumber 
\end{eqnarray}
where
\begin{equation*}
\lambda(x)=x(x+\gamma+\delta+1)
\end{equation*}
and
\begin{equation*}
\alpha+1=-N \textrm{ or } \beta+\delta+1=-N \textrm{ or }
\gamma+1=-N, \textrm{ with } N \textrm{ a non-negative integer}.
\end{equation*}

We note that the Racah polynomials are terminating Saalsch\"utzian
${}_4F_3(1)$ hypergeometric series.

The special case $\alpha = \beta = \gamma + \delta = 0, \gamma=T$, where
$T$ is a positive integer leads to the definition of
\begin{equation}
\label{250}
R_n(s,T):
=R_n(\lambda(s);0,0,T,-T)
={}_4F_3 \left[ {\displaystyle -n,n+1,-s,s+1;
\atop \displaystyle 1,1-T,1+T;} 1\right],
\end{equation}
where $0 \leq n,s \leq T-1$.

It is conjectured by Kresch and Tamvakis \cite{KT} that
\begin{equation}
\label{260}
|R_n(s,T)| \leq 1
\end{equation}
for all $0 \leq n,s \leq T-1, T \geq 1$. 
Special cases of the conjecture are proven by Kresch
and Tamvakis in \cite{KT}. Special cases of the conjecture
are also proven by Ismail and Simeonov \cite{IS}.
Furthermore, Ismail and Simeonov demonstrate asymptotics
for $R_n(s,T)$
in \cite{IS} that are in agreement with the conjecture.

\section{Main result}

\begin{Definition}
\label{D310}
Let $\{x_n\}_{n=0}^{\infty} \subseteq \mathbb{C}$. For each $n \geq 0$, we define
\begin{equation}
\label{310}
\tilde{x}_n=\sum_{k=0}^n (-1)^k \binom{n}{k} \binom{n+k}{k} x_k.
\end{equation}
\end{Definition}

We note that the transformation (\ref{310}) is linear.

In view of the formulas
\begin{equation}
\label{312}
(-1)^k \binom{n}{k} = \frac{(-n)_k}{k!}
\end{equation}
and
\begin{equation}
\label{313}
\binom{n+k}{k} = \frac{(n+1)_k}{k!},
\end{equation}
equation (\ref{310}) can also be written as
\begin{equation}
\label{314}
\tilde{x}_n=\sum_{k=0}^n \frac{(-n)_k(n+1)_k}{k!k!} x_k.
\end{equation}

We remark that the transformation in Definition \ref{D310} is inspired by
the binomial transform (introduced by Knuth in \cite{Knuth})
of a sequence $\{x_n\}_{n=0}^{\infty} \subseteq \mathbb{C}$
defined by
\begin{equation}
\label{316}
\hat{x}_n=\sum_{k=0}^n (-1)^k \binom{n}{k} x_k,
\quad n \geq 0.
\end{equation}
The inversion formula for the binomial transform is well-known 
(see \cite{Rior}) and is
\begin{equation}
\label{316}
x_n=\sum_{k=0}^n (-1)^k \binom{n}{k} \hat{x}_k,
\quad n \geq 0.
\end{equation}

Certain terminating hypergeometric series can be considered as
binomial transforms. For example, the Chu-Vandermonde formula (\ref{220}) can be
written as
\begin{equation*}
\sum_{k=0}^n (-1)^k \binom{n}{k} \frac{(a)_k}{(b)_k}
=\frac{(b-a)_n}{(b)_n},
\end{equation*}
and therefore we can conclude that the binomial transform of the sequence
$\left\{\frac{(a)_n}{(b)_n}\right\}_{n=0}^{\infty}$ is the sequence
$\left\{\frac{(b-a)_n}{(b)_n}\right\}_{n=0}^{\infty}$.

\begin{Theorem}
\label{T310}
Let $\{x_n\}_{n=0}^{\infty} \subseteq \mathbb{C}$. 
Let $\{\tilde{x}_n\}_{n=0}^{\infty} \subseteq \mathbb{C}$
be defined by (\ref{310}).
Then for each $n \geq 0$, we have
\begin{equation}
\label{320}
x_n=\sum_{k=0}^n (-1)^k \frac{(2k+1) \binom{n}{k}}{(n+k+1) \binom{n+k}{k}} \tilde{x}_k.
\end{equation}
\end{Theorem}

Equation (\ref{320}) gives us the inverse transformation of the transformation defined in (\ref{310}).

Using the formulas (\ref{312}) and (\ref{313}), we can also write equation (\ref{320}) as
\begin{equation}
\label{330}
x_n=\sum_{k=0}^n \frac{(2k+1)(-n)_k}{(n+1)_{k+1}} \tilde{x}_k,
\quad n \geq 0.
\end{equation}

Before we prove Theorem \ref{T310}, we need the following lemma:

\begin{Lemma}
\label{L310}
For each $n \geq 1$, we have
\begin{equation}
\label{340}
\sum_{k=0}^{n-1} (-1)^k (2k+1) \binom{2n+1}{n-k} = (-1)^{n-1} (2n+1).
\end{equation}
\end{Lemma}

\begin{proof}
The result is directly verified when $n=1,2,3,4$. Assume $n \geq 5$.
We let
$$A=\sum_{k=0}^{n-1} (-1)^k (2k+1) \binom{2n+1}{n-k}.$$
We split off the first two terms and the last two terms of $A$ to get
\begin{eqnarray*}
&&A=\binom{2n+1}{n} - 3\binom{2n+1}{n-1} 
+\sum_{k=2}^{n-3} (-1)^k (2k+1) \binom{2n+1}{n-k} \\
&&+(-1)^{n-2}(2n-3)\binom{2n+1}{2}
+(-1)^{n-1}(2n-1)\binom{2n+1}{1}.
\end{eqnarray*}
For $0 \leq k \leq n-2$, we apply (\ref{225}) twice to
$\binom{2n+1}{n-k}$ and obtain
\begin{eqnarray*}
&&\binom{2n+1}{n-k}
=\binom{2n}{n-k}+\binom{2n}{n-k-1} \\
&&=\binom{2n-1}{n-k}+2\binom{2n-1}{n-k-1}+\binom{2n-1}{n-k-2}.
\end{eqnarray*}
Also, by (\ref{225}),
\begin{equation*}
\binom{2n+1}{1}=\binom{2n}{1}+\binom{2n}{0} 
=\binom{2n-1}{1}+\binom{2n-1}{0}+\binom{2n}{0}.
\end{equation*}
Therefore,
\begin{eqnarray*}
&&A=\binom{2n-1}{n}+2\binom{2n-1}{n-1}+\binom{2n-1}{n-2} \\
&&-3\left(\binom{2n-1}{n-1}+2\binom{2n-1}{n-2}+\binom{2n-1}{n-3}\right) \\
&&+\sum_{k=2}^{n-3} (-1)^k (2k+1) 
\left(\binom{2n-1}{n-k}+2\binom{2n-1}{n-k-1}+\binom{2n-1}{n-k-2}\right) \\
&&+(-1)^{n-2}(2n-3)\left(\binom{2n-1}{2}+2\binom{2n-1}{1}+\binom{2n-1}{0}\right) \\
&&+(-1)^{n-1}(2n-1)\left(\binom{2n-1}{1}+\binom{2n-1}{0}+\binom{2n}{0}\right). \\
\end{eqnarray*}
From here, we can write
$$A=A_1+A_2+A_3,$$
where
\begin{equation*}
A_1=\binom{2n-1}{n}+2\binom{2n-1}{n-1}-3\binom{2n-1}{n-1}, 
\end{equation*}
\begin{eqnarray*}
&&A_2=\binom{2n-1}{n-2} 
-3\left(2\binom{2n-1}{n-2}+\binom{2n-1}{n-3}\right) \\
&&+\sum_{k=2}^{n-3} (-1)^k (2k+1) 
\left(\binom{2n-1}{n-k}+2\binom{2n-1}{n-k-1}+\binom{2n-1}{n-k-2}\right) \\
&&+(-1)^{n-2}(2n-3)\left(\binom{2n-1}{2}+2\binom{2n-1}{1}\right) \\
&&+(-1)^{n-1}(2n-1)\binom{2n-1}{1},
\end{eqnarray*}
and
\begin{equation*}
A_3=(-1)^{n-2}(2n-3)\binom{2n-1}{0}
+(-1)^{n-1}(2n-1)\left(\binom{2n-1}{0}+\binom{2n}{0}\right).
\end{equation*}

Since $\binom{2n-1}{n}=\binom{2n-1}{n-1}$, we have that $A_1=0$.

To evaluate $A_2$, we have
\begin{eqnarray*}
&&A_2=\sum_{k=0}^{n-3}(-1)^k(2k+1)\binom{2n-1}{n-k-2} \\
&&+\sum_{k=1}^{n-2}(-1)^k(2k+1)(2)\binom{2n-1}{n-k-1} 
+\sum_{k=2}^{n-1}(-1)^k(2k+1)\binom{2n-1}{n-k} \\
&&=\sum_{k=1}^{n-2}(-1)^{k-1}(2k-1)\binom{2n-1}{n-k-1} 
+\sum_{k=1}^{n-2}(-1)^k(4k+2)\binom{2n-1}{n-k-1} \\
&&+\sum_{k=1}^{n-2}(-1)^{k+1}(2k+3)\binom{2n-1}{n-k-1} 
=0,
\end{eqnarray*}
where the last equality follows by combining the three sums into
one and factoring out 
$(-1)^{k-1}\binom{2n-1}{n-k-1}$.

Finally,
\begin{equation*}
A_3=(-1)^{n-2}(2n-3)+(-1)^{n-1}(4n-2) 
=(-1)^{n-1}(2n+1).
\end{equation*}

Therefore,
\begin{equation*}
A=A_1+A_2+A_3
= (-1)^{n-1} (2n+1),
\end{equation*}
which completes the proof.
\end{proof}

\begin{proof}[Proof of Theorem \ref{T310}]
We will prove that equation (\ref{330}) holds for
each $n \geq 0$. In our proof, we will use the equivalent form 
(\ref{314}) of (\ref{310}).

For $n,m \geq 0$, we define
\begin{equation*}
a_{n,m}=
\frac{(2m+1)(-n)_m}{(n+1)_{m+1}}.
\end{equation*}
We note that $a_{n,m}=0$ for $m>n$, since
$(-n)_m=0$ for $m>n$.
In order to prove the theorem, we need to show
that for each $m \geq 0$, the transformation (\ref{314})
of the sequence $\{a_{n,m}\}_{n=0}^{\infty}$
is given by
\begin{equation}
\label{350}
\tilde{a}_{n,m}=\left\{
\begin{array}{ll}
1, & n=m,\\
0, & n \neq m.
\end{array} \right.
\end{equation}

When $m=0$, we have
$\{a_{n,0}\}_{n=0}^{\infty}
=\{\frac{1}{n+1}\}_{n=0}^{\infty}$.
Using the Chu-Vandermonde formula
(\ref{220}), we have
\begin{eqnarray*}
&&\tilde{a}_{n,0}
=\sum_{k=0}^n \frac{(-n)_k(n+1)_k}{k!k!} \frac{1}{k+1} \\
&&=\sum_{k=0}^n \frac{(-n)_k(n+1)_k}{k!(2)_k} 
={}_2F_1 \left[ {\displaystyle -n,n+1;
\atop \displaystyle 2;} 1\right] \\
&&=\frac{(1-n)_n}{(2)_n}
= \left\{ 
\begin{array}{ll}
1, & n=0,\\
0, & n \neq 0.
\end{array} \right. 
\end{eqnarray*}
Therefore (\ref{350}) holds for $m=0$.
Assume now that (\ref{350}) holds for all
$m=0,1,\ldots,p$, for some $p \geq 0$. We will show that
(\ref{350}) holds for $m=p+1$ and the result will 
follow by induction.

For $n \geq 0$, we define
\begin{equation*}
b_{n,p+1}=\frac{1}{(n+1)_{p+2}}.
\end{equation*}
Using the Chu-Vandermonde formula
(\ref{220}), we have
\begin{eqnarray*}
&&\tilde{b}_{n,p+1}
=\sum_{k=0}^n \frac{(-n)_k(n+1)_k}{k!k!} \frac{1}{(k+1)_{p+2}} \\
&&=\frac{1}{(p+2)!}\sum_{k=0}^n \frac{(-n)_k(n+1)_k}{k!(p+3)_k} 
=\frac{1}{(p+2)!} {}_2F_1 \left[ {\displaystyle -n,n+1;
\atop \displaystyle p+3;} 1\right] \\
&&=\frac{1}{(p+2)!} \frac{(p+2-n)_n}{(p+3)_n}
=\frac{(p+2-n)_n}{(p+n+2)!}.
\end{eqnarray*}
When $n>p+1$, we have that $(p+2-n)_n=0$ and so
$\tilde{b}_{n,p+1}=0$. When $n=p+1$, we have
\begin{equation*}
\tilde{b}_{p+1,p+1}=\frac{(1)_{p+1}}{(2p+3)!}
=\frac{1}{(p+2)_{p+2}}.
\end{equation*}
Using the induction hypothesis and the fact that
the transformation (\ref{314}) is linear, it follows that if
we define
\begin{equation*}
c_{n,p+1}=(p+2)_{p+2}b_{n,p+1}
-\sum_{m=0}^p \frac{(p+2)_{p+2}(p+2-m)_m}{(p+m+2)!}a_{n,m},
\quad n \geq 0,
\end{equation*}
then we will have
\begin{equation*}
\tilde{c}_{n,p+1}=\left\{
\begin{array}{ll}
1, & n=p+1,\\
0, & n \neq p+1.
\end{array} \right.
\end{equation*}

It remains to show that
$c_{n,p+1}=a_{n,p+1}$ for all $n \geq 0$.

We compute
$$\frac{(p+2)_{p+2}(p+2-m)_m}{(p+m+2)!}
=\frac{(2p+3)!}{(p+1-m)!(p+m+2)!}
=\binom{2(p+1)+1}{p+1-m}.$$
Hence for each $n \geq 0$,
\begin{equation}
\label{355}
c_{n,p+1}=\frac{(p+2)_{p+2}}{(n+1)_{p+2}}
-\sum_{m=0}^p \binom{2(p+1)+1}{p+1-m} \frac{(2m+1)(-n)_m}{(n+1)_{m+1}}.
\end{equation}
By combining all terms under a common denominator, it follows that we
can write $c_{n,p+1}=\frac{f(n)}{(n+1)_{p+2}}$, where $f(n)$
is a polynomial in $n$ of degree at most $p+1$. Now since
$\tilde{c}_{n,p+1}=0$ for $n=0,1,\ldots,p$, we must have
$c_{n,p+1}=0$ for $n=0,1,\ldots,p$. But then 
$f(n)=0$ for $n=0,1,\ldots,p$ and so we must have
$f(n)=\alpha\prod_{i=0}^p(n-i)$ for some 
$\alpha \in \mathbb{R}$. 
In view of (\ref{355}),
$$\alpha=-\sum_{m=0}^p (-1)^m (2m+1) \binom{2(p+1)+1}{p+1-m}.$$
By Lemma \ref{L310}, 
$\alpha=(-1)^{p+1} (2(p+1)+1)$. Therefore, 
\begin{eqnarray*}
&&c_{n,p+1}
=\frac{(-1)^{p+1} (2(p+1)+1)\prod_{i=0}^p(n-i)}{(n+1)_{p+2}} \\
&&=\frac{(2(p+1)+1)(-n)_{p+1}}{(n+1)_{p+2}}
=a_{n,p+1} \mbox{ for all } n \geq 0,
\end{eqnarray*}
which shows that
(\ref{350}) holds for $m=p+1$ and completes the proof by induction.  
\end{proof}

Interesting consequences of Theorem \ref{T310} are given below:

\begin{Corollary}
\label{C310}
We have the following identities:
\begin{enumerate}
\item[$(a)$] For $0 \leq m \leq n$, 
\begin{equation}
\label{360}
\sum_{k=m}^n (-1)^k
\frac{(2k+1)\binom{n-m}{k-m}}{(n+k+1)\binom{n+k}{m+k}}
=\left\{
\begin{array}{ll}
(-1)^n, & m=n,\\
0, & 0 \leq m < n.
\end{array}
\right.
\end{equation}
\item[$(b)$] For $n \geq 0$,
\begin{equation}
\label{363}
\sum_{k=0}^n (-1)^k \binom{n}{k} \binom{n+k}{k} \frac{1}{2k+1}
=\frac{1}{2n+1}.
\end{equation}
\end{enumerate}
\end{Corollary}

\begin{proof}
$(a)$ Using Theorem \ref{T310} and then switching the order of summation, 
we have that for every $n \geq 0$, 
\begin{eqnarray*}
&&x_n=\sum_{k=0}^n (-1)^k \frac{(2k+1) \binom{n}{k}}{(n+k+1) \binom{n+k}{k}} \tilde{x}_k \\
&&=\sum_{k=0}^n \left( (-1)^k \frac{(2k+1) \binom{n}{k}}{(n+k+1) \binom{n+k}{k}}
\left( \sum_{m=0}^k (-1)^m \binom{k}{m} \binom{k+m}{m} x_m \right) \right) \\
&&=\sum_{m=0}^n \left( (-1)^m \left(
\sum_{k=m}^n (-1)^k
\frac{(2k+1)\binom{n}{k}\binom{k}{m}\binom{k+m}{m}}
{(n+k+1)\binom{n+k}{k}} \right) x_m \right) \\
&&=\sum_{m=0}^n \left( (-1)^m \binom{n}{m}^2 
\left( \sum_{k=m}^n (-1)^k
\frac{(2k+1)\binom{n-m}{k-m}}{(n+k+1)\binom{n+k}{m+k}}
\right) x_m \right),
\end{eqnarray*}
where the last equality follows from (\ref{230}) and (\ref{233}). 
Hence we must have
\begin{equation*}
(-1)^m \binom{n}{m}^2 
\left( \sum_{k=m}^n (-1)^k
\frac{(2k+1)\binom{n-m}{k-m}}{(n+k+1)\binom{n+k}{m+k}}
\right)
=\left\{
\begin{array}{ll}
1, & m=n,\\
0, & 0 \leq m < n,
\end{array} \right.
\end{equation*}
and this implies (\ref{360}).

$(b)$ It is enough to show that the sequence $\{\frac{1}{2n+1}\}_{n=0}^{\infty}$ is
fixed by the inverse transformation of (\ref{310}). Indeed, we have
\begin{eqnarray*}
&&\sum_{k=0}^n \frac{(2k+1)(-n)_k}{(n+1)_{k+1}} \frac{1}{2k+1}
=\sum_{k=0}^n \frac{(-n)_k}{(n+1)_{k+1}}
=\frac{1}{n+1}\sum_{k=0}^n \frac{(-n)_k}{(n+2)_k} \\
&&\frac{1}{n+1}{}_2F_1 \left[ {\displaystyle -n,1;
\atop \displaystyle n+2;} 1\right] 
=\frac{1}{n+1}\frac{(n+1)_n}{(n+2)_n}
=\frac{1}{2n+1},
\end{eqnarray*}
where in the next-to-last step we used the Chu-Vandermonde
formula (\ref{220}).
\end{proof}

The special case $m=0$ in (\ref{360}) gives
\begin{equation}
\label{365}
\sum_{k=0}^n (-1)^k
\frac{(2k+1)\binom{n}{k}}{(n+k+1)\binom{n+k}{k}}
=\left\{
\begin{array}{ll}
1, & n=0,\\
0, & n>0.
\end{array}
\right.
\end{equation}
In view of (\ref{330}), we can also write (\ref{365}) as
\begin{equation}
\label{366}
\sum_{k=0}^n \frac{(2k+1)(-n)_k}{(n+1)_{k+1}} 
=\left\{
\begin{array}{ll}
1, & n=0,\\
0, & n >0.
\end{array}
\right.
\end{equation}

We note that (\ref{363}) implies that the sequence $\{\frac{1}{2n+1}\}_{n=0}^{\infty}$ is
fixed by the transformation (\ref{310}). In fact, since in the last term of the sum
$$\tilde{x}_n=\sum_{k=0}^n (-1)^k \binom{n}{k} \binom{n+k}{k} x_k$$
the coefficient in front of $x_n$ is 
$(-1)^n \binom{2n}{n} \neq 1$ for $n>0$, it follows that,
up to constant multiples, the sequence $\{\frac{1}{2n+1}\}_{n=0}^{\infty}$ is
the only one fixed by the transformation (\ref{310}).

\begin{Corollary}
\label{C320}
Let $0 \leq s \leq T-1$. Then for every $m$ such that
$0 \leq m \leq T-1$, we have
\begin{equation}
\label{369}
\sum_{n=0}^m (-1)^n
\frac{(2n+1)\binom{m}{n}}{(m+n+1)\binom{m+n}{n}}
R_n(s,T)=\frac{(-s)_m(s+1)_m}{(1-T)_m(1+T)_m}.
\end{equation}
\end{Corollary}

\begin{proof}
Let
$$x_n=\left\{
\begin{array}{ll}
\frac{(-s)_n(s+1)_n}{(1-T)_n(1+T)_n}, & 0 \leq n \leq T-1\\
0, & n>T-1.
\end{array}\right.$$
Then for $0 \leq n \leq T-1$,
\begin{eqnarray*}
&&R_n(s,T)={}_4F_3 \left[ {\displaystyle -n,n+1,-s,s+1;
\atop \displaystyle 1,1-T,1+T;} 1\right] \\
&&=\sum_{k=0}^n (-1)^k
\binom{n}{k}\binom{n+k}{k}
\frac{(-s)_k(s+1)_k}{(1-T)_k(1+T)_k}
=\tilde{x}_n
\end{eqnarray*}
Theorem \ref{T310} now yields the result.
\end{proof}

In view of (\ref{330}), we can also write (\ref{369}) as
\begin{equation}
\label{372}
\sum_{n=0}^m \frac{(2n+1)(-m)_n}{(m+1)_{n+1}} 
R_n(s,T)=\frac{(-s)_m(s+1)_m}{(1-T)_m(1+T)_m},
\end{equation}
for all $m$ such that $0 \leq m \leq T-1$.

We note that
$\frac{(-s)_m(s+1)_m}{(1-T)_m(1+T)_m}=0$
if $s+1 \leq m \leq T-1$, and so we have
\begin{equation}
\label{375}
\sum_{n=0}^m (-1)^n
\frac{(2n+1)\binom{m}{n}}{(m+n+1)\binom{m+n}{n}}
R_n(s,T)=0,
\quad s+1 \leq m \leq T-1,
\end{equation}
or, equivalently,
\begin{equation}
\label{378}
\sum_{n=0}^m \frac{(2n+1)(-m)_n}{(m+1)_{n+1}} 
R_n(s,T)=0,
\quad s+1 \leq m \leq T-1.
\end{equation}

\end{document}